\documentclass[11pt]{amsproc}%
\usepackage{amsmath}
\usepackage{amssymb}
\usepackage{amsthm}
\def\leq {\leqslant}

\def\le {\leqslant}
\def\ge {\geqslant}
\def\geq {\geqslant}

\newcommand{\sgn}{\mathop{\rm sgn}\nolimits}

\DeclareMathOperator{\sign}{sign}
\DeclareMathOperator{\sinc}{sinc}
\DeclareMathOperator{\dist}{dist}
\DeclareMathOperator{\supp}{supp}

\newcommand{\norm}[1]{ \left\lVert#1\right\rVert}

\usepackage{amssymb}
\usepackage{amsfonts}
\usepackage{amsmath}
\usepackage{hyperref}
\usepackage{graphicx}%
\RequirePackage{srcltx}

\usepackage{amssymb}

\providecommand{\U}[1]{\protect\rule{.1in}{.1in}}
\theoremstyle{plain}

\newtheorem{theorem}{Theorem}[section]
\newtheorem{lemma}[theorem]{Lemma}
\newtheorem{definition}[theorem]{Definition}

\newtheorem{remark}[theorem]{Remark}
\newtheorem{proposition}[theorem]{Proposition}

\numberwithin{equation}{section}

\newtheorem{corollary}[theorem]{Corollary}

\newcommand\R{\mathbb{R}}
\newcommand\N{\mathbb{N}}
\newcommand\Z{\mathbb{Z}}

\newcommand\g{{\gamma}}
\newcommand\be{{\beta}}
\newcommand\sa{{\sigma}}
\newcommand\la{{\lambda}}
\newcommand\iy{{\infty}}
\newcommand\TT{\mathcal{T}}
\newcommand{\vphi}{{\varphi}}
\newcommand{\bna}{\begin{eqnarray}}
\newcommand{\ena}{\end{eqnarray}}
\newcommand{\ba}{\begin{eqnarray*}}
\newcommand{\ea}{\end{eqnarray*}}
\newcommand{\beq}{\begin{equation}}
\newcommand{\eeq}{\end{equation}}
\def\R{{\Bbb R}}

\def\leq {\leqslant}
\def\le {\leqslant}
\def\ge {\geqslant}
\def\geq {\geqslant}
\def\R {\mathbb R}

\def\Z {\mathbb Z}
\def \N{\mathbb N}

\usepackage[T2A]{fontenc}
\usepackage[cp1251]{inputenc}
\usepackage[english]{babel}
\usepackage{color}
 \RequirePackage{srcltx}

 \usepackage{color}

\numberwithin{equation}{section}

\begin{document}
\title[Bernstein--Nikolskii inequalities]
{
Bernstein--Nikolskii inequalities: \\optimality with respect to the smoothness parameter}

\author{Michael I. Ganzburg}
\address{M. I. Ganzburg, 212 Woodburn Drive, Hampton, VA 23664 }
\email{michael.ganzburg@gmail.com}

\author{Miquel Saucedo}
\address{M. Saucedo,
 Centre de Recerca Matem\`{a}tica\\
Campus de Bellaterra, Edifici C
08193 Bellaterra (Barcelona), Spain;
}

\email{miquelsaucedo98@gmail.com}

\author{Sergey Tikhonov}
\address{S. Tikhonov,
ICREA, Pg. Llu\'{i}s Companys 23, 08010 Barcelona, Spain\\
Centre de Recerca Matem\`{a}tica,
Campus de Bellaterra, Edifici C
08193 Bellaterra (Barcelona), Spain,
 and Universitat Aut\`{o}noma de Barcelona,
 Edifici C
08193 Bellaterra (Barcelona), Spain}
\email{ stikhonov@crm.cat}

\thanks{M. Saucedo is supported by  the Spanish Ministry of Universities through the FPU contract FPU21/04230.
 S. Tikhonov is supported by PID2023-150984NB-I00 and 2021 SGR 00087. This work is supported by
the CERCA Programme of the Generalitat de Catalunya, and the Severo Ochoa and Mar\'ia de Maeztu
Program for Centers and Units of Excellence in R\&D (CEX2020-001084-M)
}

\begin{abstract}
In this paper, we study the form of the constant $C$ in the Bernstein--Nikolskii inequalities
$\|f^{(s)}\|_q
\lesssim C(s, p, q)\left\|f\right\|_p,\,0<p<q \leq\infty$,
for trigonometric polynomials and entire functions of exponential type.
We obtain the optimal behavior of the constant with respect to the smoothness parameter $s$.
\end{abstract}

\maketitle

\footnotetext{Dedicated to Karlheinz Gr\"{o}chenig on the occasion of his 65th birthday.}

\section{Introduction}
\noindent

\subsection{Notation and Preliminaries}\label{S1.1}
Let us first define the $L_p$- and $\ell_p$-quasinorms
\ba
\begin{array}{lll}
\|f\|_p:=\left(\int_\R\vert f(x)\vert^pdx\right)^{1/p},
&\|f\|_p^*:=\left(\int_{-\pi}^\pi\vert f(x)\vert^pdx\right)^{1/p},
&p\in(0,\iy];\\
 \|a\|_{\ell_p(\Z)}:=\left(\sum_{k\in\Z}\left\vert a_k\right\vert^p\right)^{1/p},
 &\|\la\|_{\ell_p(\Z_+)}:=\left(\sum_{d=0}^\iy\left\vert \la_d\right\vert^p\right)^{1/p},
 &p\in(0,\iy].
\end{array}
\ea
The Fourier transform of an integrable function $\phi$ and its inverse Fourier transform are defined by
the standard formulae
\ba
\widehat{\phi}(\xi):=\int_\R\phi(x)e^{-ix\xi}dx,\qquad
\Check{\phi}(\xi):=\frac{1}{2\pi}\int_\R\phi(x)e^{ix\xi}dx.
\ea

Next, let $\TT_n$ be the space of all trigonometric polynomials
$T_n(x)=\sum_{k=-n}^nc_ke^{ikx}$ of degree at most $n$  with complex coefficients.
We also use  the Dirichlet kernel $D_N(x):=1+2\sum_{k=1}^N \cos k x$ of order $N\in\N$
on several occasions.

In addition, let $E_{\sa,p}$ be the Paley--Wiener space of all complex-valued entire functions
$f$ of exponential type
$\sa>0$ with $\|f\|_{p}<\iy$
(recall that an entire complex-valued function $f$ is of exponential type $\sa$ if for any $\varepsilon>0$ there exists
a constant $C_{\varepsilon,f}$ such that $|f(z)|\leq C_{\varepsilon,f} e^{(\sa+\varepsilon)|z|}$ for all complex $z$).
It is well known that $E_{\sa,p}\subset E_{\sa,q}$, if $0<p< q\le\iy$,
    and $E_{\sa,p}$
    is a Banach space for $p\in[1,\iy]$ and a quasi-Banach space for $p\in(0,1)$.
    Throughout the paper, we use the notation $E_p:=E_{1,p}$.

    In the paper, we often use without referencing
    the following fundamental result by Paley and Wiener:
    $f \in E_2$ if and only if there exists $\psi$ with
     $\norm{\psi}_2<\iy$ such that $\psi$ is supported in $[-1,1]$
     and $f(\xi)=\widehat{\psi}(\xi)$.

    The sinc function $\sinc x:=\sin x/x$,
    the floor function $\lfloor x\rfloor$,
     the ceiling function $\lceil x\rceil$,
      the gamma function $\Gamma(z)$,
     and the beta function
 \beq\label{E3.4a}
     B(z_1,z_2)
     :=\int_0 ^1 t^{z_1-1} (1-t)^{z_2-1} dt
      = \frac{\Gamma(z_1) \Gamma(z_2)}{\Gamma(z_1+z_2)},
     \qquad z_1>0,\quad z_2>0,
 \eeq
 are used as well.
 The following asymptotic formula for the quotient of the gamma functions is well known
 (see \cite[Sect. 1.18, Eqn. (5)]{EMOT1953}):
 \beq\label{E3.4b}
 \lim_{z\to \iy}z^{b-a}\frac{\Gamma(z+a)}{\Gamma(z+b)}=1,\qquad z>0,\quad a,b\in\R.
 \eeq

We also use the notations
$\lesssim,\,\gtrsim$, and $\asymp$ in the following sense:
$
\vphi(\tau,\gamma,\ldots)\lesssim \delta(\tau,\gamma,\ldots)$,
$\vphi(\tau,\gamma,\ldots)\gtrsim \delta(\tau,\gamma,\ldots),$
and
 $\vphi(\tau,\gamma,\ldots)\asymp \delta(\tau,\gamma,\ldots)$
mean that there exist constants $C_1>0$ and/or $C_2>0$ independent of the essential parameters
$\tau,\,\gamma,\ldots$ such that
$\vphi(\tau,\gamma,\ldots)\leq C_1 \delta(\tau,\gamma,\ldots),\,
\vphi(\tau,\gamma,\ldots)\geq C_2 \delta(\tau,\gamma,\ldots)$,
and $C_2\delta(\tau,\gamma,\ldots)\le
\vphi(\tau,\gamma,\ldots)\le C_1\delta(\tau,\gamma,\ldots),$ respectively, for the relevant ranges of $\tau,\gamma,\ldots$.
The dependence of the constants on certain parameters is indicated
by using subscripts, e.g., $\lesssim_{p,q}$ or $\asymp_{p,q}$. The absence of the subscripts means that
$C_1$ and\slash{or} $C_2$ are absolute constants.

\subsection{Bernstein--Nikolskii Inequalities}\label{S1.2}
The following
Bernstein--Nikolskii inequalities:
\bna
&&\left\|T_n^{(s)}\right\|^*_q
\lesssim_{s,p,q} n^{s+1/p-1/q}\left\|T_n\right\|^*_p,\qquad T_n\in\TT_n;\label{E1.1}\\
&&\left\|f^{(s)}\right\|_q
\lesssim_{s,p,q} \sa^{s+1/p-1/q}\left\|f\right\|_p,\qquad f\in E_{\sa,p};\label{E1.2}
\ena
for $0<p\le q\le\iy$ and $s=0,\,1,\ldots$,
have been known for more than sixty years.

Both \eqref{E1.1} and \eqref{E1.2} are fundamental results in pure and applied harmonic analysis.
The applications include various topics in approximation theory,  function analysis, and PDE's.

Various estimates and asymptotics for the constants in \eqref{E1.1} and \eqref{E1.2}
were discussed in numerous publications (see, e.g., \cite{G2005, DT2005, SG2015,
LL2015, GT2017,
ca2019,
BH2020, G2021, BCOS2022}
and references therein).
In particular, the sharp constants in  \eqref{E1.1} and \eqref{E1.2} are known for
$q=\iy$ and $p=2$ (see \cite[Eqns. (1.5) and (1.13)]{GT2017}).
The first and third authors \cite{GT2017} discussed limit relations between the sharp constants
in  \eqref{E1.1} and \eqref{E1.2}.
Bang and Huy \cite[Theorem 2.1]{BH2020} studied the influence of a large $s$ on the constant in \eqref{E1.2}.
Simonov and Glazyrina \cite[Theorem 1]{SG2015} studied the sharp constant in the closely related to \eqref{E1.1} Markov-Nikolskii inequality  for $q=\iy$ and $p=1$.
In the case $s=0$ and $p<q=\infty$, estimate  \eqref{E1.2}  was recently investigated in \cite{LL2015,ca2019, BCOS2022}.

Importantly,  the problem of finding the asymptotic behavior of  the constant in \eqref{E1.1} as $n\to\iy$ remains open.

In this paper, we obtain new estimates of the constants in \eqref{E1.1} and \eqref{E1.2}
that cannot be improved on the spaces $\TT_n$ and $E_{\sa,p}$
with respect to the parameter $s\in\N$.

\subsection{Main Results}\label{S1.3}
\begin{theorem}\label{T1.1} For every $n,s\in\mathbb{N}$ and $0< p<q\leq \infty$, 
 the following relation holds:
 \beq\label{E1.3}
\sup_{\substack{T_n\in\mathcal{T}_n,\\ \norm{T_n}_p^*=1}}\norm{ T_n^{(s)}}_q^* \asymp_{p,q} n^s \left(1+ \left({n}/{s}\right)^{{1}/{p}-{1}/{q}}\right).
\eeq
\end{theorem}

\begin{theorem}\label{T1.2} For every $s\in\mathbb{N}$ and $0< p<q\leq \infty$, 
 the following relation holds:
 \beq\label{E1.4}
\sup_{\substack{f\in E_{\sa,p},\\ \norm{f}_p=1}}\norm{ f^{(s)}}_q \asymp_{p,q}  \sa^s(\sa/ s)^{{1}/{p}-{1}/{q}}.
\eeq
\end{theorem}

For a certain subset of $\TT_n$, relation \eqref{E1.3}
for $q=\iy$ and $p=1$ can be improved.

\begin{definition}\label{D1.3}
     We say that a finite sequence of numbers $c=(c_0, \dots, c_{n})$ is concave if  $c_0\geq c_1 \geq \dots \geq c_n \geq  c_{n+1}:=0$ and
 $c_n-c_{n+1} =:\Delta c_n \geq \Delta c_{n-1} \geq \dots \geq \Delta c_0=c_0-c_1$.
 \end{definition}

 The next theorem shows that by restricting ourselves just to concave sequences of coefficients, we come to the better behavior in $s$ and $n$ on the right-hand side of \eqref{E1.3}.
 Let $\TT_{n,\mbox{c}}\subset\TT_{n}$ be the set of all trigonometric polynomials $T_n(x)=c_0+\sum_{k=1}^n 2c_k \cos{kx}$
 with a concave sequence $c$.
\begin{theorem}\label{T1.4}
   For every $n, s\in\mathbb{N}$, 
 the following relation holds:
 \beq\label{E1.5}
\sup_{\substack{T_n\in\mathcal{T}_{n,\mbox{c}},\\ \norm{T_n}_1^*=1}}\norm{ T_n^{(s)}}_\iy^* \asymp \left(n^s+\frac{n^{s+1}}{s+1} \right)\left(\frac1{\log{(s+2)}}+ \frac1{\log(n+2)}\right).
\eeq
\end{theorem}

\begin{remark}\label{R1.5}
 A weighted version of the $\lesssim$ part of Theorem \ref{T1.2} for $1\le p<q\le\iy$
 with power weights was established
by Bang and Huy in \cite[Theorem 2.1]{BH2020}.
\end{remark}

\begin{remark}\label{R1.66}
Analyzing the proofs, we note that
all theorems above are valid for the fractional case, namely, when $s\in(0,\iy)$. Here, the derivatives of a trigonometric polynomial are defined in the Weyl sense, i.e., $T^{(s)}(x):=\sum_{|k|\le n} (ik)^s c_k e^{ikx}$,
 where $(ik)^s:= |k|^s e^{i\frac{\pi}{2} s\sgn k}$. The similar definition for functions $f\in E_{\sigma,p}$ can be found in \cite{SKM1993}.
\end{remark}

The proofs  of the lower and upper estimates in \eqref{E1.3} and \eqref{E1.4} are presented in Sections
\ref{S2} and \ref{S3}, respectively.
The properties of concave sequences and trigonometric polynomials from $\TT_{n,\mbox{c}}$,
which are needed for the proof of Theorem \ref{T1.4}, are discussed in Section \ref{S4}.
Section \ref{S5} contains the proofs of all theorems.
In Section \ref{S6}, we discuss several properties of extremal trigonometric polynomials in
the  Bernstein--Nikolskii inequality for   $p=1$ and $q=\infty$.
We note that it suffices to prove Theorem \ref{T1.2} for $\sa=1$, so we discuss functions
from $E_p,\,p\in(0,\iy]$, in Sections \ref{S2}, \ref{S3}, \ref{S5}, and \ref{S7}.

For $p\in [1,\iy]$, the $\lesssim$ part of Theorem \ref{T1.1}
is proved by the standard method of estimation of
$\norm{D_n^{(s)}}_r^*$.
The proof of the $\lesssim$ part of Theorem \ref{T1.1} for $p\in (0,1)$ is more complicated.
It is based on an inequality for functions from $E_p$ that
is obtained by using certain properties of the discrete Hardy spaces, which are discussed
in the Appendix of Section \ref{S7}.
The proof of the $\lesssim$ part of Theorem \ref{T1.2} for $p\in (0,\iy]$ uses the
$\lesssim$ part of Theorem \ref{T1.1}
and also Theorem 1.4 from \cite{GT2017}.
The lower estimates in Theorems \ref{T1.1} and \ref{T1.2} are proved by the construction of special functions from
$E_{p}$ and $\TT_n$ in the spirit of \cite[Sect. 84]{A1965}.

Finally, it is worth mentioning that the following improvements of the  $\lesssim$ parts of relations
 \eqref{E1.3} and \eqref{E1.4} are valid in the case $p\in (0,1]$ and $q=\infty$:
 for any $T_n(x)=\sum_{k=-n}^nc_ke^{ikx}$ and  any $f\in E_p$,
\ba
\sum_{k=-n}^n \vert k\vert^s \vert c_k\vert
\lesssim_p  n^s \left( 1+ \left({n}/{s}\right)^{1/p}\right) \norm{T_n}_p^*, \qquad
\int_{-1}^1 \vert\xi\vert^{s} \vert\widehat{f}(\xi)\vert \,d\xi\lesssim_p   s^{-1/p} \norm{f}_p
\ea
 (see Lemmas  \ref{L3.4} and \ref{L3.5}, respectively).

\vspace{.2cm}
\section{Lower Bounds in Theorems \ref{T1.1} and \ref{T1.2}}\label{S2}
In this section, we prove the $\gtrsim$ parts of Theorems \ref{T1.1} and \ref{T1.2}.
\begin{lemma}\label{L2.1}
    For every $s\in\mathbb{N}$ and $0< p<q\leq \infty$, 
the following relation holds:
\beq\label{E2.1}
\sup_{\substack{f\in E_p,\\ \norm{f}_p=1}}\norm{ f^{(s)}}_q \gtrsim_{p,q}  {s}^{{1}/{q}-{1}/{p}}.
\eeq
\end{lemma}
\begin{proof}
    Let $\phi$ be a non-negative bump function, that is, a smooth function supported in $[-1,1]$ such that $\phi(x)= 1$ for $|x|\leq {1}/{2}$. Let us set
    \ba
    \phi_s(x):= s \phi\left( s x-s+1\right),\qquad s\ge 2.
    \ea
    Since $\phi_s$ is a smooth function supported in $[1- {2}/{s},1]$, its Fourier transform $\widehat{\phi}_s$ is also smooth and it belongs to $E_p$ for any $p\in(0,\iy]$.
    Moreover,
    $
   \left|\widehat{\phi}_s(\xi)\right|= \left|\widehat{\phi}\left(s^{-1}\xi\right)\right|.
   $
   Therefore,
    \beq\label{E2.2}
    \norm{\widehat{\phi}_s}_p =\left(
    \int_{\mathbb{R}} \left|\widehat{\phi}\left(s^{-1}\xi\right)\right|^p d\xi\right)^{1/p}
    = \norm{\widehat{\phi}}_ps^{1/p}.
    \eeq
Furthermore, if $|\xi | \leq s$, then

              \ba
              \left\vert\left(\widehat{\phi}_s\right)^{(s)}(\xi)\right\vert
              &=&  \left\vert\int_{1-2/ s}^{1} s \phi\left ( s x-s+1\right) e^{-i x\xi}\, x ^s dx\right\vert=
              \\
              &=& \left\vert\int_{-1/s}^{1/s} s \phi\left ( sx\right) e^{-i x\xi}\, \left(x+1-{1}/{s}\right) ^s dx\right\vert\\
              &\geq& (1-1/s)^s \int_{0}^{1/s} s \phi\left ( sx\right)\cos x\xi  dx
                \gtrsim  \int_{0}^{1/s} s \phi\left ( sx\right)  dx
                > 1/2.
              \ea
Hence,
              \beq\label{E2.3}
              \norm{\left(\widehat{\phi}_s\right)^{(s)}}_q
              \geq \left(\int_{0}^s
              \left\vert\left(\widehat{\phi}_s\right)^{(s)}(\xi)\right\vert^qd\xi\right)^{1/q}
              \gtrsim s^{1/q}.
              \eeq
               Finally, \eqref{E2.1} follows from \eqref{E2.2} and \eqref{E2.3}.
\end{proof}

For trigonometric polynomials, we use a variation of the previous argument, with the Jackson kernel taking the place of the bump function.

\begin{lemma}\label{L2.2}
For every $n,s\in\mathbb{N}$ and $0< p<q\leq \infty$, 
 the following relation holds:
 \beq\label{E2.4}
\sup_{\substack{T_n\in\mathcal{T}_n,\\ \norm{T_n}_p^*=1}}\norm{ T_n^{(s)}}_q^* \gtrsim_{p,q} n^s \left( 1+ \left({n}/{s}\right)^{{1}/{p}-{1}/{q}}\right).
\eeq
\end{lemma}
\begin{proof}
Note first that
\beq\label{E2.5}
\sup_{\substack{T_n\in\mathcal{T}_n,\\ \norm{T_n}_p^*=1}}\norm{ T_n^{(s)}}_q^*
\geq \norm{\left(e^{in\cdot}\right)^{(s)}}_q^*/\norm{e^{in\cdot}}_p^*\gtrsim_{p,q} n^s.
\eeq
 Next, let $r=r(p)\in \mathbb{N}$ be the smallest number such that $pr>1$.
 Then \eqref{E2.5} shows that if $n\leq 4rs$, \eqref{E2.4} holds.

Furthermore, assuming that $n>4rs$ for $s\ge 2$ and  setting
$N:=\lfloor {n/rs} \rfloor \geq 4$, we consider the Jackson kernel
\ba
J_{r,N}(x):= \left(\frac{\sin{(N+1/2)x}}{\sin{x/2}}\right)^{r}
\ea
with
\beq\label{E2.6}
\norm{J_{r,N}}_p^*= \norm{D_N}_{rp}^{*r} \asymp_{p} N^{r-1/p}.
\eeq
Note that (\ref{E2.6}) easily follows  from the Hardy--Littlewood equivalence  for monotone Fourier coefficients; namely, for $f(x)\sim\sum_{k=0}^\iy a_k \cos kx$ with a non-increasing $\{a_k\}_{k=0}^\iy$ we have $\|f\|_p^p\asymp_p \sum_{k=0}^\iy a_k^p (1+k)^{p-2}$ (see \cite[Ch. XII, Lemma 6.6]{Z2002}).

  The Fourier coefficients $\hat{J}_{r,N}(\cdot)$
   of $J_{r,N}$ are the convolution of the coefficients of $J_{1,N}$ with themselves $r$ times. Since $\hat{J}_{1,N}(m)=1$ for $|m|\leq N$ and $0$ otherwise, we deduce that $\hat{J}_{2,N}(m) \asymp N$ for $|m|\leq N$. Likewise, we obtain that $\hat{J}_{3,N}(m) \asymp N^2$ for $|m|\leq N/2$.  By induction on $r$, we further obtain that the coefficients $\hat{J}_{r,N}(\cdot)$ are non-negative and there is a $\lambda=\lambda(r)\in(0,1)$ such that for $N$ large enough,
   \beq\label{E2.7}
  \widehat{J}_{r,N}(m)\asymp_{r} N^{r-1},\qquad |m|\leq \lambda N.
  \eeq
  Note that \eqref{E2.7} for an even $r$ can be also proved by the following asymptotic relation
  (see, e.g., \cite[Ch. 2, Eqn. (3.38)]{DS2008}):
  $1-\pi \widehat{J}_{r,N}(m)/\norm{J_{r,N}}_1^*\asymp_r (m/N)^2$.

Next,  setting
\ba
    T(x):= J_{r,N}(x)e^{i(n-r N)x},
\ea
    we see that $T\in\TT_n$  and by \eqref{E2.6},
     \beq\label{E2.8}
    \norm{T}_p^* \asymp_{p} N^{r-1/p}\asymp_{p} \left({n}/{s}\right)^{r-1/p}.
\eeq
    Then it follows from \eqref{E2.7} that for $|x|\leq {\pi}/{(4rN)} \asymp_p s/n$ we have
     \ba
\vert T^{(s)}(x)\vert
&\geq &\left | \sum _{m=-rN}^{rN}\widehat{J}_{r,N}(m)(n-rN+m)^s\cos m x\right|\\
&\gtrsim&\sum _{\vert m\vert\leq \lambda N}\widehat{J}_{r,N}(m)(n-rN+m)^s\\
 &\gtrsim_p& (n-rN)^s N^{r}
 \geq (1-1/s)^s n^s N^r
  \gtrsim_p n^s \left({n}/{s} \right)^{r}.
\ea
  Hence,
\beq\label{E2.9}
\norm{T^{(s)}}_q^*
\geq \left(\int_{\vert x\vert\leq \pi/(4rN)}\left\vert T^{(s)}(x)\right\vert^q dx\right)^{1/q}
\gtrsim_{p,q} n ^s \left({n}/{s} \right)^{r-1/q},
\eeq
and combining \eqref{E2.8} with \eqref{E2.9}, we obtain
 \beq\label{E2.10}
    \sup_{\substack{T_n\in\mathcal{T}_n,\\ \norm{T_n}_p^*=1}}\norm{ T_n^{(s)}}_q^* \gtrsim_{p,q} n^s  \Big({n}/{s}\Big)^{{1}/{p}-{1}/{q}}.
    \eeq
    Finally, \eqref{E2.4} for $n>4rs$ and $s\ge 2$
     follows from \eqref{E2.10} and \eqref{E2.5}.
     This completes the proof of the lemma.
\end{proof}

\vspace{.2cm}
\section{Upper Bounds in Theorem \ref{T1.1}}\label{S3}
In this section, we prove the $\lesssim$ part of Theorem \ref{T1.1}.

\subsection{The case $p\in[1,\iy)$}\label{S3.1}

\begin{lemma} \label{L3.1} If $n,s\in\N$ and $r\in(1,\infty],$ then
\beq\label{E3.1}
  \norm{D_n^{(s)}}_r^*\lesssim _r n^s\left(1+\left({n}/{s}\right)^{1-{1}/{r}}\right).
\eeq
\end{lemma}

\begin{proof}
 The inequality is trivial for $r=\infty$. If $r\in(1,\iy)$,
then using again the Hardy--Littlewood equivalence  for monotone Fourier coefficients
\cite[Ch. XII, Lemma 6.6]{Z2002}, we obtain
\beq\label{E3.2}
\norm{D_n^{(s)}}_r^*
\lesssim \norm{\sum_{k=1}^n k^s e^{ik\cdot}}_r^*
= \norm{\sum_{k=1}^n (n-k+1)^s e^{ik\cdot}}_r^*
\lesssim_r
 \left(\sum_{k=1}^n k^{r-2} (n-k+1)^{rs} \right)^{{1}/{r}}.
\eeq
Next, we have by elementary estimates and by \eqref{E3.4a} and \eqref{E3.4b},
\bna\label{E3.3}
&&\sum_{k=3}^{n} \left(\frac{k}{n}\right)^{r-2}
\left(1-\frac{k-1}{n}\right)^{rs}
    \lesssim_r n\sum_{k=3}^{n}  \int_{(k-2)/n}^{(k-1)/n}t^{r-2} (1-t)^{rs}dt\\
    &&< n\int_0 ^1 t^{r-2} (1-t)^{rs} dt
    \lesssim_r n {\Gamma(rs+1)}/{\Gamma(rs+r)}
    \lesssim_r ns^{1-r}.
    \nonumber
\ena
Combining now \eqref{E3.2} with \eqref{E3.3}, we arrive at
\ba
\norm{D_n^{(s)}}_r^*
\lesssim_r \left(n^{rs}+2^{r-2}(n-1)^{rs}+n^{r-1+rs}s^{1-r}\right)^{1/r}
\lesssim_r n^s+n^{s+1-1/r}s^{1/r-1}.
\ea
Thus \eqref{E3.1} is established.
\end{proof}


As a corollary of Lemma \ref{L3.1}, we obtain the $\lesssim$ part of Theorem \ref{T1.1} for $p\in[1,\iy)$.
\begin{corollary}\label{C3.2} For $n,s\in\N$ and $1\leq p<q\leq\infty$,
\ba
\sup_{\substack{T_n\in\mathcal{T}_n,\\ \norm{T_n}_p^*=1}}\norm{ T_n^{(s)}}_q^*
\lesssim_{p,q} n^s \left(1+ \left({n}/{s}\right)^{{1}/{p}-{1}/{q}}\right).
\ea
\end{corollary}
\begin{proof} Using Young's inequality for periodic convolutions and Lemma \ref{L3.1}, we obtain
for any $T_n\in\TT_n$
\ba
    \norm{T_n^{(s)}}_q^*
    \lesssim \norm{D_n^{(s)}}_r^* \norm{T_n}_p^*
    \lesssim_{p,q} n^s\left(1+\left({n}/{s}\right)^{1-{1}/{r}}\right)\norm{T_n}_p^*,
\ea
where $0\le 1/r:=1/q-1/p+1<1$.
\end{proof}

\subsection{The case $p\in(0,1]$}\label{S3.2}
Here, we prove the $\lesssim$ part of Theorem \ref{T1.1} for $p\in(0,1]$.
Note that the case $p=1$ is also discussed in Corollary \ref{C3.2}.
Our main result is the following:
          \begin{lemma}\label{L3.5}
              If $s\in\N$ and $p\in(0,1]$, then for any $f\in E_p$,
              \beq\label{E3.6}
              \norm{f^{(s)}}_\iy
                 \lesssim \int_{-1}^1 \vert\xi\vert^{s} \vert\widehat{f}(\xi)\vert\,d\xi
                 \lesssim_p   s^{-1/p} \norm{f}_p.
          \eeq
          \end{lemma}

\begin{remark}
    Note that for $p=1$, the right relation of \eqref{E3.6} can be easily obtained as follows:
    \ba
    \int_{-1}^1 |\xi|^s |\widehat{f}(\xi) | \,d \xi
    \leq \norm{f}_1 \int_{-1}^1 |\xi|^s  \,d \xi
    \lesssim s^{-1} \norm{f}_1.
    \ea
    However, the proof for $p\in (0,1)$ is more complicated,
and it requires certain properties of the discrete Hardy spaces.
\end{remark}
The counterpart of Lemma \ref{L3.5} for trigonometric polynomials is given in the next result.

\begin{lemma}\label{L3.4}
If $n,s\in\N$ and $p\in(0,1]$, then for any
$T_n(x)=\sum_{k=-n}^n c_k e^{ik x}\in \TT_n$,
\beq\label{E3.5}
\norm{T_n^{(s)}}_\iy^*
\leq\sum_{k=-n}^n \vert c_k\vert \vert k\vert^s
\lesssim_p  n^s \left( 1+ \left({n}/{s}\right)^{1/p}\right) \norm{T_n}_p^*.
\eeq
\end{lemma}
Combining together Lemmas \ref{L3.5} and \ref{L3.4} with Bernstein's inequalities for entire functions of exponential type $1$ and trigonometric polynomials, we obtain Theorems \ref{T1.2} and \ref{T1.1}
for  $0<p<q\le\iy$ and $p\in(0,1]$.
\begin{corollary}\label{C3.5a}
If $s\in\N$ and $0<p<q\le \iy,\,p\in(0,1]$,
              then, for any $f\in E_p$,
              \ba
              \norm{ f^{(s)}}_q \lesssim_{p}  s^{{1}/{p}-{1}/{q}}\norm{ f}_p.
              \ea
              \end{corollary}

\begin{proof}
Using Bernstein's inequality (see (1.2) for $p = q$) and relation
\eqref{E3.6}, we obtain
\ba
\norm{ f^{(s)}}_q
\leq \norm{ f^{(s)}}_\iy^{1-p/q}\norm{ f^{(s)}}_p^{p/q}
\lesssim_{p}  s^{{1}/{q}-{1}/{p}}\norm{ f}_p.
\ea
\end{proof}

\begin{corollary}\label{C3.3} If $n,s\in\N$ and $0<p<q\leq\infty,\,p\in(0,1]$,
then for any $T_n\in\TT_n$,
\ba
\norm{ T_n^{(s)}}_q^*
\lesssim_{p} n^s \left(1+ \left({n}/{s}\right)^{{1}/{p}-{1}/{q}}\right)\norm{ T_n}_p.
\ea
\end{corollary}
The proof is similar to that of Corollary \ref{C3.5a}.

We first prove  Lemma \ref{L3.5} for special functions from $E_p$.

          \begin{lemma}\label{L3.6}
              Let $s\in\N$ and $p\in(0,1]$. In addition, let $a$ be an $H_p$-atom as defined in
               Definition \ref{D7.2a}, and let $f_a$ be defined by \eqref{E7.1a}. Then
            \beq\label{E3.7}
                 \int_{-1}^1 \vert\xi\vert^{s} \vert\widehat{f_a}(\xi)\vert\,d\xi \lesssim_p   s^{-1/p} \norm{f_a}_p.
          \eeq
          \end{lemma}

          \begin{proof}
Let the interval $I:=[-N+M,N+M],\,N\in\N,\,M\in\Z$, contain the support of $a$. It follows from the properties of the $H_p$-atoms
(see Definition \ref{D7.2a})
 that $\norm{a}_\infty \lesssim N^{-1/p}$
              and $\sum_{k=-N}^N a_{k+M} (k+M)^j=0$ for any integer $j$ such that  $0\leq j \leq p^{-1}-1$.
              It is easy to show by induction that
              $\sum_{k=-N}^N a_{k+M} k^j=0$ for $0\leq j \leq p^{-1}-1$.

              In addition, $a\in H_p(\Z)$ by Theorem \ref{T7.3}.
Therefore,
              $f_a(x)= \sum_{k=-N+M}^{N+M} (-1)^k a_k \sinc({x-k\pi})$ belongs to $E_p$
              by Theorem \ref{T7.2}
              and $\left\|f_a\right\|_1<\iy$ by \eqref{E1.2}. Then we have
              \ba
              \widehat{f_a}(\xi)
               =(-1)^M e^{-iM\pi \xi}\sum_{k=-N}^N (-1)^{k} a_{k+M} e^{-ik\pi \xi}
               \left\{\begin{array}{lll}
               \pi, &\vert\xi\vert<1,\\
               \pi/2, &\vert\xi\vert=1,\\
               0, &\vert\xi\vert>1.
               \end{array}\right.
               \ea
              Furthermore, setting
              \ba
              g(\xi):=\pi \sum_{k=-N}^N (-1)^{k} a_{k+M} e^{-ik\pi\xi},\qquad j_0:=\lfloor p^{-1} -1\rfloor,
              \ea
we observe that the following properties hold:
              \begin{enumerate}
                  \item[1)]  $|g(\xi)|=|\widehat{f}_a(\xi)|,\,\xi\in(-1,1);$
                  \item[2)]  $|g(\xi)| \lesssim N^{1-1/p},\,\xi\in\R;$
                  \item[3)] $g^{(j)}(1)=(-i)^j\pi^{j+1}\sum_{k=-N}^N a_{k+M} k^j=0$
                  for $0\leq j \leq j_0$;
                  \item[4)]
                  $\left\vert g^{(j)}(\xi)\right\vert \lesssim N^{j+1-1/p}$
                  for $\xi\in\R$ and $j\geq 0$;
                  \item[5)] $|g(\xi)|\lesssim (1-\xi)^{j_0+1} N^{j_0+2-1/p}
                  ,\,\xi\in[-1,1].$
                   \end{enumerate}

                  Note that the last property follows from property 4) and from
                  Taylor's theorem for $g$ at $\xi=1$ with the Lagrange form of the remainder
                  since the first $j_0$ derivatives of $g$ vanish at $\xi=1$
                  by property 3).
             To conclude the proof of the lemma, we show that
             \beq\label{E3.8}
             \int_0 ^1 \xi^{s} |{g}(\xi)| d \xi \lesssim_p s^{-1/p}.
             \eeq
              The estimate
              \beq\label{E3.9}
               \int_{-1} ^0 |\xi|^{s} |{g}(\xi)| d \xi\lesssim_p s^{-1/p}.
              \eeq
              can be established similarly.

             If $N>s$, then using property 2), we deduce that
             \ba
                 \int_0 ^1 \xi^{s} |{g}(\xi)| d \xi\lesssim_p   N^{1-1/p} s^{-1} \leq s^{-1/p}.
                \ea
            If $N\leq s$, then property 5) and relations \eqref{E3.4a} and \eqref{E3.4b} yield
                 \ba
                 \int _{0}^1 \xi^s |{g}(\xi)| d \xi
                 \lesssim N^{j_0+2-1/p}\int_{0} ^1  \xi^s (1-\xi)^{j_0+1}  d \xi
                 &\lesssim_p&N^{j_0+2-1/p} \Gamma(s+1)/\Gamma\left(s+j_0+3\right)\\
                 &\lesssim_p& N^{j_0+2-1/p} s^{-2-j_0} \leq s^{-1/p}.
                 \ea
                 Hence, \eqref{E3.8} and analogously \eqref{E3.9} follow. Finally, combining \eqref{E3.8} and \eqref{E3.9} with property 1), we arrive at \eqref{E3.7}.
                 \end{proof}

Having \eqref{E3.6} for functions $f_a$ in hand,
we use the atomic structure of the discrete Hardy spaces given in Theorem \ref{T7.3} to extend Lemma \ref{L3.6} to any function $f$.

\begin{proof}[Proof of Lemma \ref{L3.5}]
By Theorem \ref{T7.2}, for $f\in E_p$ there exists a unique $a\in H_p(\Z)$ such that
$f=f_a$ with $\norm{f}_p \asymp_p \norm{a}_{H_p(\Z)}$. Furthermore, by Theorem      \ref{T7.3},  $a$ has an atomic decomposition $a=\sum_{d=0}^\infty \lambda _d a(d)$, where $a(0),\,a(1),\ldots,$ are $H_p$-atoms,
 with $\norm{a}_{H_p(\Z)} \asymp_p \norm{\lambda}_{\ell_p(\Z_+)}$.

              Then, $f_a=\sum_{d=0}^\infty \lambda_d f_{a(d)}$ and since $\left\|f_a\right\|_1<\iy$ by \eqref{E1.2},
              we also have
              $\widehat{f_a}=\sum_{d=0}^\infty \lambda_d \widehat{f}_{a(d)}$.
              Therefore, using Lemma \ref{L3.6}, we conclude that
              \ba
              \int_{-1}^1 |\xi|^s |\widehat{f}_a(\xi)| d\xi
              &\leq& \int_{-1}^1 |\xi|^s \sum_{d=0}^\infty |\lambda _d| |\widehat{f}_{a(d)}(\xi)| d\xi
              =  \sum_{d=0}^\infty |\lambda_{d}| \int_{-1}^1 |\xi|^s   |\widehat{f}_{a(d)}(\xi)|d\xi
              \\
              &\lesssim_p& s^{-1/p} \sum_{d=0}^\infty |\lambda _d|
              \leq s^{-1/p} \norm{\lambda}_{\ell_p(\Z_+)}
              \lesssim_p s^{-1/p}  \norm{f}_p.
              \ea
\end{proof}
We now use Lemma \ref{L3.5} to obtain Lemma \ref{L3.4} by a discretization argument.
        \begin{proof}[Proof of Lemma \ref{L3.4}]

           Let $T_n(x)=\sum_{k=-n}^n c_k e^{ik x}$ and let $\phi$ be a non-negative bump function supported in $[-1,1]$ such that $\phi(\tau)= 1$ for $|\tau|\leq {1}/{2}$.

           If $2n+1\geq s$, let us define a function $f$ by the following formula for its Fourier transform:
        \beq\label{E3.13}
           \widehat{f}(\xi):=(2n+1)\sum_{k=-n}^n c_{k} \phi\left((2n+1)\xi - 2k\right).
            \eeq
            Then  $\phi\left((2n+1)\xi - 2k\right)$
            is supported in $I_k:=\left[\frac{2k-1}{2n+1},\frac{2k+1}{2n+1}\right], \,-n\le k\le n$,
            $\widehat{f}$ is supported in $[-1,1]$, and moreover,
            \beq\label{E3.14}
            \widehat{f}(\xi)=(2n+1)c_{k} \phi\left((2n+1)\xi - 2k\right),  \qquad \xi\in I_k,\quad
            -n\le k\le n.
            \eeq
            Next, applying the inverse Fourier transform to both sides of
            \eqref{E3.13}, we observe that the function
\ba
f(x)=\Check{\phi}\left(\frac{x}{2n+1}\right)\sum_{k=-n}^n c_{k} e^{\frac{2ikx}{2n+1}}= \Check{\phi}\left(\frac{x}{2n+1}\right)T_n\left(\frac{2x}{2n+1}\right)
\ea
 belongs to $E_p$ and
\beq\label{E3.15}
      \norm{f}_p\lesssim_p n^{{1}/{p}} \norm{T_n}_p^*.
      \eeq
      Indeed,
      \ba
      \norm{f}_p
      =(n+1/2)^{1/p}\norm{\Phi_p T_n}_p^*
      \leq (n+1/2)^{1/p}\norm{\Phi_p}_\iy^* \norm{T_n}_p^*,
      \ea
      where
      $\Phi_p(y):=\left(\sum_{k\in\Z}\left\vert\Check{\phi}(y/2+k\pi)\right\vert^p\right)^{1/p},\,y\in[-\pi,\pi)$.

     Furthermore, using \eqref{E3.14} and
     taking into account that $(n+1/2)^s\lesssim n^s$ for $2n+1\geq s$, we obtain
     \bna\label{E3.16}
     \int_{-1}^1 \vert\xi\vert^s \vert\widehat{f}(\xi)\vert\,d\xi
     &=&(2n+1)\int_{-1}^1|\xi|^s  \left\vert\sum_{k=-n}^n c_k   \phi\left((2n+1)\xi -2k\right)\right\vert d \xi
     \\
     &=& (2n+1)\sum_{k=-n}^n |c_k|\int_{I_k}  |\xi|^s   \phi\left((2n+1)\xi -2k\right) d \xi
     \nonumber\\
     &\geq& (2n+1)\sum_{k=-n}^n |c_k|
     \frac{\left\vert (k+1/4)^{s+1}- (k-1/4)^{s+1}\right\vert}{(s+1)(n+1/2)^{s+1}}
     \nonumber\\
     &\gtrsim& n^{-s}\sum_{k=-n}^n |c_k||k|^s.
     \nonumber
     \ena
      Therefore, using Lemma \ref{L3.5} and relations \eqref{E3.15} and \eqref{E3.16}, we conclude that
 \beq\label{E3.17}
\sum_{k=-n}^n |c_k||k|^s \lesssim_p n^{s+1/p} s^{-1/p} \norm{T_n}_p^*,\qquad 2n+1\geq s.
\eeq

If $2n+1<s$ and $s\ge 2$, we proceed by using a modification of the previous argument.
     Let
       \ba
       \widehat{f}(\xi):=s\sum_{k=-n}^n c_k \phi\left(s\xi - (s-1){k}/{n}\right).
       \ea
    Then
    $\phi\left(s\xi - (s-1){k}/{n}\right)$ is supported
    in $I_k:=\left[-\frac{1}{s}+\frac{(s-1)k}{ns},\frac{1}{s}+\frac{(s-1)k}{ns}\right]$, $-n\le k\le n$,
    $\widehat{f}$ is supported in $[-1,1]$, and moreover,
            \beq\label{E3.17a}
            \widehat{f}(\xi)=s c_k \phi\left(s\xi - (s-1){k}/{n}\right),\qquad \xi\in I_k,\quad -n\le k\le n,
            \eeq
    since the family $\left\{ I_k\right\}_{k=-n}^n$  is pairwise disjoint
    by the assumption $2n+1<s$.
    Next, the function
      \ba
      f(x)=\Check{\phi}(x/s)\sum_{k=-n}^n c_k e^{(1-s^{-1})\frac{ikx}{n}}
          = \Check{\phi}\left(\frac{x}{s}\right)T_n\left(\frac{\left(1-s^{-1}\right)x}{n}\right)
      \ea
     belongs to $E_p$ and
     \beq\label{E3.18}
    \norm{f}_p\lesssim_p s^{{1}/{p}} \norm{T_n}_p^*.
    \eeq
      Indeed, note that  the polynomial $T_n\left((2 \pi/\alpha) x\right)$ is $\alpha$-periodic with $\alpha:={2\pi n}/{(1-s^{-1})}$, and since $\Check{\phi}$ is rapidly decaying, it satisfies the estimate
      $\left|\Check{\phi}(x/s)\right|\lesssim_p {(1+|x/s|)^{-2/p}}$.
     Then we have
     \bna\label{E3.18a}
     \norm{f}_p
     &\lesssim_p& \left(\int_\R (1+|x/s|)^{-2}\left|T_n((2\pi/\alpha)x)\right|^p\right)^{1/p}\\
     \nonumber
     &\leq&
     \left(\sum_{j=0}^\iy
     \frac1{(1+j)^{2}}
     \int_{js}^{(j+1)s}\left(\left|T_n((2\pi/\alpha)x)\right|^p
     +\left|T_n(-(2\pi/\alpha)x)\right|^p\right)dx
     \right)^{1/p}\\
     \nonumber
     &\leq&
     \left(\sum_{j=0}^\iy
     \frac1{(1+j)^{2}}
     \int_{\alpha \lfloor js/\alpha\rfloor}
     ^{\alpha \lceil (j+1)s/\alpha\rceil}\left(\left|T_n((2\pi/\alpha)x)\right|^p
     +\left|T_n(-(2\pi/\alpha)x)\right|^p\right)dx
     \right)^{1/p}\\
     \nonumber
     &=& \left(\sum_{j=0}^\iy
     \frac{\alpha}{(1+j)^{2}}
     \int_{ \lfloor js/\alpha\rfloor}
     ^{ \lceil (j+1)s/\alpha\rceil}\left(\left|T_n(2\pi x)\right|^p
     +\left|T_n(-2\pi x)\right|^p\right)dx
     \right)^{1/p}\\
     \nonumber
     &\lesssim_p&
   \left(
   \sum_{j=0}^\iy
     \frac{1}{(1+j)^{2}}\right)^{1/p}
     \left(
  (2\alpha + s)
   \int_{0}^{1}\left|T_n( 2\pi x)\right|^pdx
     \right)^{1/p}.
     \ena
Since $\alpha< \pi s$, we see  that \eqref{E3.18a} implies \eqref{E3.18}.
    Finally, using \eqref{E3.17a}, we obtain
       \bna\label{E3.19}
       \int_{-1}^1 \vert\xi\vert^s \vert\widehat{f}(\xi)\vert\,d\xi
       &=& s\int_{-1}^1  |\xi|^s  \left|\sum_{k=-n}^n c_k \phi\left(s(\xi - (1-s^{-1}){k}/{n})\right) \right| d \xi
       \\
       &=& s\sum_{k=-n}^n |c_k|\int_{I_k}  |\xi|^s  \phi\left(s(\xi - (1-s^{-1}){k}/{n})\right) d \xi
       \nonumber\\
       &\geq& \frac{s}{s+1}\sum_{k=-n}^n |c_k|
       \left\vert\left(\frac{(s-1)k}{s n}+\frac{1}{2s}\right)^{s+1}
       -\left(\frac{(s-1)k}{s n}-\frac{1}{2s}\right)^{s+1}\right\vert
       \nonumber\\
     &\geq& \sum_{k=-n}^n |c_k| \left ((1-s^{-1}){|k|}/{n}\right)^s
     \gtrsim n^{-s}\sum_{k=-n}^n |c_k||k|^s.
     \nonumber
     \ena
     Therefore, using Lemma \ref{L3.5} and relations \eqref{E3.18} and \eqref{E3.19}, we conclude that
     \beq\label{E3.20}
     \sum_{k=-n}^n |c_k||k|^s \lesssim_p n^{s} \norm{T_n}_p^*,\qquad 2n+1< s.
     \eeq
     Thus, the right relation of
     \eqref{E3.5} follows from \eqref{E3.17} and \eqref{E3.20}.
     Lemma \ref{L3.4} and Corollary \ref{C3.3} are established.
        \end{proof}

\vspace{.2cm}
\section{Properties of Concave Sequences and Polynomials from $\TT_{n,\mbox{c}}$ }\label{S4}

Throughout this section, $c=(c_0, \ldots, c_{n})$ is a concave sequence (see Definition \ref{D1.3}).
The proof of Theorem \ref{T1.4} is based on several technical lemmas and on  Proposition \ref{P4.4}, all of which are of independent interest.

Let us define the following  sequences:

\bna \label{E4.3}
   &&v_l(k):=\min\left\{1- \frac{k-1}{n}, 1- \frac{l-1}{n}\right\},
\qquad l =0,\ldots,n,\quad k =0,\ldots,n; \\
  &&V_l:=(v_l(0),\ldots,v_l(n)),\qquad l =0,\ldots,n.
  \nonumber
\ena
The sequences $V_l,\, l =0,\ldots,n$, are concave. Moreover,
it turns out that every concave sequence can be represented as a linear combination of
$V_l,\, l =0,\ldots,n$, with nonnegative coefficients.

\begin{lemma}\label{L4.2}
    If $c$ is concave, then there exist numbers $h_l\geq 0, \, l=0,\ldots,n$, such that
    $c=\sum_{l=0}^n h_l V_l$.
\end{lemma}
\begin{proof}
    Since
    $c_k=  \sum_{j=k}^n \Delta c_j$
    with $c_{n+1}:=0$ and
    $\Delta c_j= \Delta c_0+\sum_{l=1}^ {j} \left(-\Delta^2 c_{l-1}\right)$,
    we have
    $ c_k=\sum_{j=k}^n \left(\Delta c_0 +\sum_{l=1}^{j} \left(-\Delta^2 c_{l-1}\right)\right),\,
    0\leq k,j\leq n$. Hence for  $0\leq k\leq n$,
    \ba
    c_k= \Delta c_0 \,(n-k+1)   +\sum_{l=1}^{n} \left(-\Delta^2 c_{l-1}\right) \min\{n-k+1,n-l+1\}
    :=n \sum_{l=0}^n h_l v_l(k),
    \ea
   since $n-k+1 = \min\{n-k+1, n+1\}=nv_0(k)$.
\end{proof}

The next three lemmas contain  technical estimates for  concave sequences.
\begin{lemma}\label{L4.2a}
The following relations are valid:
\beq\label{E4.4a}
       \sum_{k=0}^n \frac{v_l(k)}{n-k+1}
         \asymp \left(1- \frac{l-1}{n}\right) \left(\left|\log\left(1-\frac{l-1}{n}\right)\right|+1\right),\qquad
        0\leq l\leq n;
       \eeq
       \beq\label{E4.4b}
\min_{0\leq l\leq n-1} \left(1-\frac{l}{n}\right)\left(\left\vert\log
      \left(1-\frac{l}{n}\right)\right\vert+1\right)
      =\frac{\log n+1}{n}.
\eeq
\end{lemma}

\begin{proof}
First, we note that
 the function $y(\log (1/y)+1)$ is increasing on $(0,1)$, which yields \eqref{E4.4b}.
Next, by the definition of $v_l(k)$,
       \beq\label{E4.4c}
       \sum_{k=0}^n \frac{v_l(k)}{n-k+1}
        =\left(1- \frac{l-1}{n}\right) \left(\sum_{m=n-l+1}^{n+1}\frac{1}{m}+1\right)-\frac{1}{n},
        \qquad
        0\leq l\leq n.
       \eeq
 Hence, \eqref{E4.4a} is valid for $l=0$.
       In addition,
       \beq\label{E4.4d}
       \sum_{m=n-l+1}^{n+1}\frac{1}{m}+1
       \asymp
       \left|\log\left(1-\frac{l-1}{n}\right)\right|+1,
       \qquad
        1\leq l\leq n.
       \eeq
       Thus, \eqref{E4.4a} follows from relations \eqref{E4.4b}--\eqref{E4.4d}.
\end{proof}
\begin{lemma}\label{L4.2b}
For $n,s\in\N$ we have 
    \beq\label{E4.5a}
    \sup_{\tau\in[0,1]} H_{n,s}(\tau)
    :=
    \sup_{\tau\in[0,1]} \frac{ \tau^{s+1}}{1- \log\left(1+n^{-1} -\tau\right)} \lesssim \frac1{\log (s +2)}+ \frac1{\log (n +2)}.
    \eeq
\end{lemma}
\begin{proof}
Let us set $\g(s):=(s+2)^{-1} \log \log (s+2)$.
     First, we have
    \beq\label{E4.5b}
     H_{n,s}(\tau)
     \lesssim \tau^{s+1}
     \leq \frac1{\log (s+2)},
    \qquad \tau\in \left[0, 1- \g(s)\right.\left.\right).
    \eeq
    Next, for $\tau\in \left[1- \g(s),1\right]$ we obtain
    \ba
    1-\log(1+n^{-1} -\tau)
    \geq 1-\log\left(s^{-1}\log \log (s+2) +n^{-1}\right)
    &\geq& 1-\log\left((s+2)^{-1/2} +n^{-1/2}\right)\\
      &\gtrsim& \log (\min\{n+2,s+2\}).
    \ea
    Hence
    \beq\label{E4.5c}
    H_{n,s}(\tau)
    \lesssim \frac{ 1}{1- \log(1+n^{-1} -\tau)}
  \lesssim \frac1{\log (s+2)}+ \frac1{\log (n+2)},
  \qquad \tau\in \left[1- \g(s),1\right].
  \eeq
Thus, the upper estimate in \eqref{E4.5a} follows from \eqref{E4.5b} and \eqref{E4.5c}.
\end{proof}
The next lemma discusses a relation between concave sequences.

\begin{lemma}\label{L4.3}
If $c$ is concave, then
\beq\label{E4.4}
         \frac{1}{n}\sum_{k=0}^n c_k + \int_{n^{-1}}^{\pi} x^{-1} c_{n-{\lfloor x^{-1} + 1 \rfloor}} dx
         \lesssim \sum_{k=0}^n \frac{c_k}{n-k+1}.
 \eeq

\end{lemma}

\begin{proof}
By  Lemma \ref{L4.2}, it suffices to establish \eqref{E4.4} for all $c_k=v_l(k),\,0\le k\le n$, defined
by \eqref{E4.3},
        with a fixed $l,\,0\le l\le n$.
        In addition, taking into account \eqref{E4.4a}, we see that inequality \eqref{E4.4}
        is a consequence of the following
        inequality:
        \beq\label{E4.4e}
        \frac1{n}{\sum_{k=0}^n v_l(k)}
 +\int_{n^{-1}}^{\pi} x^{-1} v_l({n-{\lfloor x^{-1} + 1 \rfloor})} dx
 \lesssim  \left(1- \frac{l-1}{n}\right) \left(\left|\log\left(1-\frac{l-1}{n}\right)\right|+1\right).
    \eeq

   To prove \eqref{E4.4e}, we first note that the estimate
    \beq\label{E4.4f}
\frac1{n}{\sum_{k=0}^n v_l(k)} \leq 2\left(1- \frac{l-1}{n}\right),\qquad 0\le l\le n,
\eeq
immediately follows from \eqref{E4.3}.
       Next, we show that for $0\le l\le n$,

        \beq\label{E4.4g}
       I_n(l):=\int_{n^{-1}}^{\pi} x^{-1} v_l({n-{\lfloor x^{-1} + 1 \rfloor})} dx
 \lesssim  \left(1- \frac{l-1}{n}\right) \left(\left|\log\left(1-\frac{l-1}{n}\right)\right|+1\right).
    \eeq
    To prove it, we need the following simple estimates for $x\in(0,\pi]$:

         \beq\label{E4.4h}
         v_l({n-{\lfloor x^{-1} + 1 \rfloor}}) \leq
         \min\left\{\frac{2}{n} + \frac{1}{n x},1- \frac{l-1}{n}\right\}
          \leq \min\left\{\frac{2\pi+1}{n x},1- \frac{l-1}{n}\right\}.
         \eeq
         If $l=0$, then using \eqref{E4.4h}, we see that $I_n(0)\lesssim 1$, that is, \eqref{E4.4g} holds.
         Furthermore, using again \eqref{E4.4h} for $1\le l\le n$, we have
\ba
   I_n(l)
     &=& \int_{n^{-1}}^{n^{-1} \left(1-\frac{l-1}{n}\right)^{-1}} x^{-1} v_l({n-{\lfloor x^{-1} + 1 \rfloor}}) dx
     \\
     &+& \int_{{n^{-1} \left(1-\frac{l-1}{n}\right)^{-1}}}^{\pi} x^{-1} v_l({n-{\lfloor x^{-1} + 1 \rfloor}}) dx
    \\
    &\lesssim&
       \left(1- \frac{l-1}{n}\right) \int_{n^{-1}}^{n^{-1} \left(1-\frac{l-1}{n}\right)^{-1}} x^{-1}  dx
       + \frac1n\int_{{n^{-1} \left(1-\frac{l-1}{n}\right)^{-1}}}^{\pi} x^{-2}  dx
       \\
       &\lesssim&  \left(1- \frac{l-1}{n}\right) \left(\left|\log(1-\frac{l-1}{n})\right|+1\right).
\ea

Therefore, \eqref{E4.4g} holds for $0\le l\le n$, and \eqref{E4.4e} follows from \eqref{E4.4f} and \eqref{E4.4g},
completing the proof.
\end{proof}

Our next result 
 describes the integral norm of a trigonometric polynomial with concave coefficients. 

\begin{proposition}\label{P4.4}
Let $c$ be concave and $T_n(x)=c_0+2\sum_{k=1}^n c_k \cos{kx}$. Then we have
\beq\label{E4.6}
    \norm{T_n}_1^* \asymp \sum_{k=0}^n \frac{c_k}{n-k+1}.
    \eeq
\end{proposition}

    The proof is based on two lemmas. The first one is a Hardy-type estimate (cf., e.g,  \cite[Theorem 8.7]{Z2002}).
\begin{lemma}\label{L4.5}
    For any $T_n(x)=c_0+2\sum_{k=1}^n c_k \cos{kx}$, we have
    \beq\label{E4.7}
    \norm{T_n}_1^* \gtrsim \left\vert\sum_{k=0}^n \frac{c_k}{n-k+1}\right\vert.
    \eeq
    \end{lemma}

    \begin{proof}

     For $c_0=1$ and positive coefficients a weaker version of inequality \eqref{E4.7} was proved by Nikolskii \cite[Theorem 1]{N1948}. The proof of \eqref{E4.7} is similar to that one.
     Since the polynomial
     \ba
     Q_{2n+1}(x)&:=&\frac{1}{2(n+1)}+2\sin\,(n+1)x\sum_{k=1}^n\frac{\sin kx}{k}\\
     &=&\frac{1}{2(n+1)}+\sum_{k=1}^n\frac{\cos kx}{n-k+1}-\sum_{k=1}^{n}\frac{\cos\, (k+n+1)x}{k}
     \ea
     is uniformly bounded on $[-\pi,\pi)$ (cf. \cite[Ch. VIII, Eqns. (1.4) and (1.6)]{Z2002}), we obtain
     \ba
     \norm{T_n}_1^* \gtrsim \left\vert\int_{-\pi}^\pi T_n(x)Q_{2n+1}(x)dx\right\vert
     \gtrsim \left\vert\sum_{k=0}^n \frac{c_k}{n-k+1}\right\vert.
     \ea
\end{proof}

 \begin{lemma}\label{L4.6}
 If $c$ is concave and $T_n(x)=c_0+2\sum_{k=1}^n c_k \cos{kx}$, then

 \bna
    &&|T_n(x)|\lesssim x^{-1} c_{n-{\lfloor x^{-1} + 1 \rfloor}}\qquad x\in (1/n,\pi),\label{E4.8}\\
     &&\norm{T_n}_1^* \lesssim \frac{1}{n}\sum_{k=0}^n c_k + \int_{n^{-1}}^{\pi} x^{-1} c_{n-{\lfloor x^{-1} + 1 \rfloor}} dx.\label{E4.9}
 \ena
    \end{lemma}
    \begin{proof}
We prove the lemma by the standard argument (see, e.g., \cite[Ch. V, Sect. 1]{Z2002}).
Since $T_n$ is a polynomial with non-increasing coefficients, the Abel transformation shows that
for $x\in (1/n,\pi)$,

\bna\label{E4.9.1}
|T_n(x)|
&\lesssim& x^{-1} \left|\sum_{k=0}^n \Delta c_k \sin{(k+1/2)x}\right|
\\
&\leq& x^{-1} \left|\sum_{k=0}^n \Delta c_k e^{ikx}\right|
+x^{-1}\left|\sum_{k=0}^n \Delta c_k e^{-ikx}\right|=: I_1(x)+I_2(x),
\nonumber
\ena
where $c_{n+1}:=0$.
Hence, setting $b_k:=\Delta c_{n-k},\,1\leq k\leq n,$ we obtain
\ba
         I_1(x)=x^{-1}\left|\sum_{k=0}^n b_k e^{-ikx}\right|
         \lesssim x^{-1} \sum_{k=0}^{\lfloor x^{-1} + 1 \rfloor} b_{k}
         +x^{-1} \left|\sum_{k=\lfloor x^{-1} + 1 \rfloor}^n b_{k}e^{-ikx}\right|.
\ea

Now, applying again the Abel transformation to the
sum $\sum_{k=\lfloor x^{-1} + 1 \rfloor}^n$ above
and using the fact that $b_k$ is non-increasing with $k$, we arrive at
\beq\label{E4.9.2}
         I_1(x)\lesssim
          x^{-1} \sum_{k=0}^{\lfloor x^{-1} + 1 \rfloor} b_{k}
         +
         x^{-2} b_{\lfloor x^{-1} + 1 \rfloor} 
         \lesssim    x^{-1} \sum_{k=0}^{\lfloor x^{-1} + 1 \rfloor} b_{k} =
           x^{-1} c_{n-{\lfloor x^{-1} + 1 \rfloor}}.
           \eeq
The similar estimate holds for $I_2(x)$. Thus, \eqref{E4.8} follows from \eqref{E4.9.1}, \eqref{E4.9.2}, and the corresponding estimate for $I_2(x)$.
To prove \eqref{E4.9}, we split the integral
$\int_0^\pi|T_n(x)|dx\leq (1/n)\|T_n\|_\iy^* +\int_{1/n}^\pi|T_n(x)|dx$
and use \eqref{E4.8}.
\end{proof}


\begin{proof}[Proof of Proposition \ref{P4.4}.]
The $\lesssim$ part of relation \eqref{E4.6} is a consequence of \eqref{E4.4} and \eqref{E4.9},
while the $\gtrsim$ part immediately follows from \eqref{E4.7}.
\end{proof}

\begin{remark}\label{R4.6a}
A weaker version of the $\lesssim$ part of Proposition \ref{P4.4} is a special case of the following Efimov's result
\cite[Theorem 1]{E1960} with a lengthy proof:
for any $T_n(x)=\sum_{k=0}^n c_k \cos{kx}$,
\ba
\norm{T_n}_1^*
\lesssim \left|c_0\right|+\frac{1}{n+1}\sum_{k=0}^{n-1}{(k+1)(n-k)}\left\vert \Delta^2c_k\right\vert
+\sum_{k=0}^{n}\frac{|c_k|}{n-k+1}.
\ea
Note that for concave coefficients, $\left\vert \Delta^2c_k\right\vert=-\Delta^2c_k$ and
\ba
\frac{1}{n+1}\sum_{k=0}^{n-1}{(k+1)(n-k)}\left\vert \Delta^2c_k\right\vert
=\frac{2}{n+1}\sum_{k=0}^nc_k\leq \sum_{k=0}^{n}\frac{c_k}{n-k+1}.
\ea
\end{remark}

\vspace{.2cm}
\section{Proofs of Theorems }\label{S5}
 Theorem \ref{T1.1} immediately follows from Lemma \ref{L2.2} and Corollaries \ref{C3.2} and \ref{C3.3}.
We now prove Theorem \ref{T1.2}.
 \begin{proof}[Proof of Theorem \ref{T1.2}]
 The $\gtrsim$ part  follows from Lemma \ref{L2.1}.
 To prove the $\lesssim$ part  of the theorem, we need the following result:

 \begin{lemma}\label{L5.1}
 (Ganzburg and Tikhonov \cite[Theorem 1.4]{GT2017})
 For $0<p<q\leq \infty$ and $s=0,\,1,\ldots$,
 the following relation holds:
 \ba
 \sup_{\substack{f\in E_p,\\ \norm{f}_p=1}}\norm{ f^{(s)}}_q
\leq
\liminf_{n\to \iy}n^{-s-1/p+1/q}
\sup_{\substack{T_n\in\mathcal{T}_n,\\ \norm{T_n}_p^*=1}}\norm{ T_n^{(s)}}_q^*.
\ea
 \end{lemma}
This lemma allows us to derive the $\lesssim$ part of Theorem \ref{T1.2} from
the $\lesssim$ part of Theorem \ref{T1.1}.
The proof of the theorem is complete.
 \end{proof}
 Note that
 the $\lesssim$ part of the theorem for $p\in(0,1]$ is proved in Corollary \ref{C3.5a}
 by a different method.

\begin{proof}[Proof of Theorem \ref{T1.4}]
    For a polynomial $T_n(\cdot)=c_0+2\sum_{k=1}^n c_k \cos{k\cdot}\in\TT_{n,\mbox{c}}$,
    we show that
        \beq\label{E5.1}
        \sum_{k=0}^n k^s c_k
         \lesssim
         \left(n^s+\frac{n^{s+1}}{s+1}\right) \left(\frac1{\log{(s+2)}}+ \frac1{\log (n+2)}\right)
          \sum_{k=0}^n \frac{c_k}{n-k+1}.
         \eeq
        Then the $\lesssim$ part of \eqref{E1.5} follows from relation \eqref{E5.1} and Proposition \ref{P4.4}. It remains to verify \eqref{E5.1}.

        By  Lemma \ref{L4.2}, it suffices to establish \eqref{E5.1} for all $c_k=v_l(k),\,0\le k\le n$,
        with a fixed $l,\,0\le l\le n$.
    Recall that by relation \eqref{E4.4a} of Lemma \ref{L4.2a},
         \beq\label{E5.2}
        \left\vert V_l\right\vert_1:=\sum_{k=0}^n \frac{v_l(k)}{n-k+1} \asymp \left(1- \frac{l-1}{n}\right) \left(\left|\log\left(1-\frac{l-1}{n}\right)\right|+1\right).
        \eeq
        Then it follows from \eqref{E5.1} and \eqref{E5.2} that it suffices to prove the following relation with a fixed $l,\,0\le l\le n$:
        \beq\label{E5.3}
        \sum_{k=0}^n k^s v_l(k)
        \lesssim
        \left(n^s+\frac{n^{s+1}}{s+1}\right) \left(\frac1{\log{(s+2)}}+ \frac1{\log (n+2)}\right)
        \left\vert V_l\right\vert_1.
        \eeq
To prove \eqref{E5.3}, we use the definition of $v_l(k)$ (see \eqref{E4.3}) and split the left-hand side of \eqref{E5.3} into two sums as follows:
        \beq\label{E5.4}
        \sum_{k=0}^n k^s v_l(k) = \left(1-\frac{l-1}{n}\right)\sum_{k=0}^{l-1} k^s
        + \sum_{k=l}^n k^s \left(1-\frac{k-1}{n}\right)=:I_1(l)+I_2(l).
        \eeq
Next, $I_1(0)=0$ and using \eqref{E5.2}
and  Lemma \ref{L4.2b} for  $n, s \in\N$
and $0\leq l\leq n$, we obtain
\bna\label{E5.5}
I_1(l)
&\leq& \left(1-\frac{l-1}{n}\right)\frac{l^{s+1}}{s+1}
\lesssim
\frac{\left\vert V_l\right\vert_1}{s+1}
\left(\frac{l^{s+1}}{1-\log\left(1-\frac{l-1}{n}\right)}\right)\\
&=& \frac{n^{s+1}\left\vert V_l\right\vert_1}{s+1}H_{n,s}(l/n)
\lesssim
\frac{n^{s+1}}{s+1} \left(\frac1{\log (s+2)}+ \frac1{\log (n+2)}\right)
\left\vert V_l\right\vert_1,
\nonumber
\ena
where $H_{n,s}$ is defined by \eqref{E4.5a}.
To estimate $I_2(l)$, note that
\beq\label{E5.6}
I_2(l)=n^{s-1}+\frac{2}{n}\sum_{k=l}^{n-1} k^s
+\sum_{k=l}^{n-1} k^s \left(1-\frac{k+1}{n}\right)
 \leq n^{s-1}+\frac{2n^s}{s+1}+\sum_{k=l}^{n-1} k^s \left(1-\frac{k+1}{n}\right).
\eeq
Furthermore, for $0\leq l\leq n-1$,
\bna\label{E5.7}
           &&\sum_{k=l}^{n-1} k^s \left(1-\frac{k+1}{n}\right)
           <n^{s+1}\int\limits_{l/n}^1t^s(1-t)dt
           \\
            &&=n^{s+1}\left(\frac{1}{s+1}-\frac{\left({l}/{n}\right)^{s+1}}{s+1}
            -\frac{1}{s+2}+\frac{\left({l}/{n}\right)^{s+2}}{s+2}\right)
            <n^{s+1}\frac{1-\left( ({l-1})/{n}\right)^{s+1}}{s(s+1)}.
            \nonumber
        \ena
        In addition, we need the following elementary inequality:
        \beq\label{E5.8}
            \frac{1-x^{s+1}}{s+1}\leq \frac{(1-x)\left(\log\frac{1}{1-x}+1\right)}{\log s},\qquad x\in [0,1) .
        \eeq
        To prove it, we first note that
        $(1-x^{s+1})/(s+1)\leq 1-x$ and, in addition, inequalities
        $x>1-1/s$ and $\log (1/(1-x))> \log s$ are equivalent.
         Hence \eqref{E5.8} is valid for  $x\in (1-1/s,1)$.
         Next, since the function $y(\log(1/y)+1)$ is increasing on $[0,1)$, we have
        \ba
            \frac{(1-x)\left(\log\frac{1}{1-x}+1\right)}{\log s}
            > \frac{1}{s}
            >\frac{1-x^{s+1}}{s+1},\qquad x\in [0,1-1/s].
        \ea
        Hence, inequality \eqref{E5.8} is established.
        Then combining \eqref{E5.7}
        and \eqref{E5.2} with \eqref{E5.8} for $x=(l-1)/n$, we obtain
        \beq\label{E5.9}
        \sum_{k=l}^{n-1} k^s \left(1-\frac{k+1}{n}\right)
        \le \frac{n^{s+1}}{s\log s} \left(1-\frac{l-1}{n}\right)\left(\left\vert\log
      \left(1-\frac{l-1}{n}\right)\right\vert+1\right)
      \lesssim \frac{n^{s+1}}{s\log s}
      \left\vert V_l\right\vert_1.
        \eeq
        The following inequalities are valid as well by relations \eqref{E5.2} and \eqref{E4.4b}:
        \beq\label{E5.10}
        {n^{s-1}}
        \lesssim
        \frac{n^{s}}{\log (n+2)}
        \left\vert V_l\right\vert_1,\qquad
        \frac{n^s}{s}
        \lesssim
        \frac{n^{s+1}}{s\log (n+2)}
        \left\vert V_l\right\vert_1.
      \eeq
      Therefore, it follows from \eqref{E5.6}, \eqref{E5.9}, and \eqref{E5.10} that
      \beq\label{E5.11}
      I_2(l)\lesssim
      \left(n^s+\frac{n^{s+1}}{s+1}\right) \left(\frac1{\log{(s+2)}}+ \frac1{\log (n+2)}\right)
        \left\vert V_l\right\vert_1.
      \eeq
      Thus, collecting together \eqref{E5.4}, \eqref{E5.5}, and \eqref{E5.11}, we arrive at \eqref{E5.3},
      eventually proving  \eqref{E5.1} and, therefore, the $\lesssim$ part of Theorem \ref{T1.4}.

      To establish the $\gtrsim$ part of the theorem, we choose
      $T_{n,l}(\cdot):=  v_l(0)+2\sum_{k=1}^nv_l(k)\cos k\cdot\in\TT_{n,\mbox{c}}$
      with a certain $l,\,0\leq l\leq n$. Recall that the concave sequence
      $V_l=(v_l(0),\ldots,v_l(n))$ is defined by \eqref{E4.3}.

      We first assume that $s\in\N$ is an even number.
       In the case of $1\leq n\leq s$, let us set $l=n$. Then by relations
        \eqref{E4.3} and \eqref{E5.2} and, in addition, by Proposition \ref{P4.4}, we obtain
   \bna\label{E5.12}
   \norm{T_{n,n}^{(s)}}_{\iy}^*
   &=&\sum_{k=1}^n k^s  v_n(k)
   > n^sv_n(n)=n^{s-1}
\gtrsim
\frac{n^{s}}{\log (n+2)}\left\vert V_n\right\vert_1 \\
 &\gtrsim&
 \left(n^s+n^{s+1}s^{-1}\right) \left( \frac{1}{\log (s+2)}+ \frac{1}{\log (n+2)}\right)
 \norm{T_{n,n}}_{1}^*.
 \nonumber
\ena

In the case of $n> s\geq 1$, let us set $l=1+\lfloor n(1-1/s)\rfloor$.
Then
\beq\label{E5.12a}
1\leq l< n,\qquad l^{s+1}\gtrsim n^{s+1},\qquad 1-(l-1)/n\geq 1/s.
\eeq
Therefore, using relations  \eqref{E4.3}, \eqref{E5.2}, and \eqref{E5.12a} and, in addition, using
Proposition \ref{P4.4}, we obtain
\bna\label{E5.13}
\norm{T_{n,l}^{(s)}}_{\iy}^*
   &=&\sum_{k=1}^n k^s  v_l(k)
   \geq \sum_{k=1}^l k^s  v_l(k)
  \\ &\gtrsim&
   \frac{\left(1-\frac{l-1}{n}\right) \frac{{l}^{s+1}}{s}}
 {
 \left(1- \frac{{l}-1}{n}\right) \left(\left(\left|\log(1-\frac{{l}-1}{n}\right)\right|+1\right)
 }
 \left\vert V_l\right\vert_1
 \gtrsim
 \frac{n^{s+1}}{s\log (s+1)}\left\vert V_l\right\vert_1
  \nonumber
 \\
 &\gtrsim&
 \left(n^{s+1}s^{-1}+n^s\right) \left( \frac{1}{\log (s+2)}+ \frac{1}{\log (n+2)}\right)
 \norm{T_{n,l}}_{1}^*.
 \nonumber
 \ena
 Thus, for an even $s$, the $\gtrsim$ part of Theorem \ref{T1.4} follows from \eqref{E5.12} and \eqref{E5.13}.

Finally, for an odd $s$, using Bernstein's inequality $\norm{T_n^{(s+1)}}_\infty^* \leq n \norm{T_n^{(s)}}_\infty^*$, we conclude that  the following relations hold:
\ba
\sup_{\substack{T_n\in\mathcal{T}_{n,\mbox{c}},\\ \norm{T_n}_1^*=1}}\norm{ T_n^{(s)}}_\iy^*
&\geq& n^{-1} \sup_{\substack{T_n\in\mathcal{T}_{n,\mbox{c}},\\ \norm{T_n}_1^*=1}}\norm{ T_n^{(s+1)}}_\iy^*\\
&\gtrsim&  \left(n^{s}+\frac{n^{s+1}}{s+1} \right)\left(\frac1{\log{(s+2)}}+ \frac1{\log (n+2)}\right).
\ea
The proof is now complete.
\end{proof}

\vspace{.2cm}
\section{Remarks on Extremal Polynomials}\label{S6}

Below we discuss several facts about the extremal trigonometric polynomials $P_n\in\TT_n$ for
  Bernstein--Nikolskii inequality \eqref{E1.1}, with   $p=1,\, q=\infty$, and  $s=0,\,1,\ldots$.
In the special case of $s=0$, the similar results are known either for trigonometric polynomials or  entire functions
of exponential type (cf. \cite{HB1993, T1965, AD2022}).

The polynomials $P_n$ are defined by the relations
\beq\label{E6.1}
P_n^{(s)}(0)=1,\qquad
\inf_{\substack{T_n\in\mathcal{T}_{n},\\ T_n^{(s)}(0)=1}}\norm{ T_n}_1^*
=\left\|P_n\right\|_1^*,
\eeq
and their existence is well known.
In addition, it is clear that by considering
$\left(\overline{P_n(x)}+ P_n(x)\right)/2$ and
$\left(P_n(x)+ (-1)^s P_n(-x)\right)/2$,
we may assume that first, $P_n$ is real-valued, and second,
it is even for even $s$ and odd for odd $s$.

    \begin{lemma}\label{L6.1}
        Let $P_n$ satisfy \eqref{E6.1}. Then for any $Q_n\in \mathcal{T}_{n}$, we have
        \beq\label{E6.2}
        \int_{-\pi}^{\pi}\sign\left(P_n(x)\right)\, Q_n(x)dx = \norm{ P_n}_1^*\,Q_n^{(s)}(0).
        \eeq
    \end{lemma}

    \begin{proof}
 For any $Q_n\in\TT_n$, the function
 \ba
\vphi(\la):=\norm{R_{n,\la}}_1^*:=
\norm{P_n + \lambda\left(Q_n -Q_n^{(s)}(0)P_n\right)}_1^*,\qquad \la\in\R,
\ea
has a minimum at $\lambda=0$ since $R_{n,\la}\in\TT_n$ and
 $R_{n,\la}^{(s)}(0)=1$. Hence, calculating $\vphi^\prime(\la)$, we arrive at \eqref{E6.2}.
    \end{proof}
    For $s=0$ and entire functions of exponential type,
    the similar fact was proved in \cite[Eqn. (2.10)]{HB1993}.
    As a simple corollary of Lemma \ref{L6.1}, we obtain the following result (cf. \cite[p. 211]{T1965} for $s=0$):
    \begin{corollary}\label{C6.2}
    Let $P_n$ satisfy \eqref{E6.1}.
        Then the best constant $\left(\norm{P_n}^{*}_1\right)^{-1}$
        in Bernstein--Nikolskii inequality \eqref{E1.1}, with   $p=1,\, q=\infty$, and $s=0,\,1,\ldots$,
        is the distance in the uniform norm from the function $h_{n,s}:=(1/(2\pi))D_n^{(s)}$ to the space $F_n$ of all functions with Fourier coefficients $c_k=0$ for $\left\vert k\right\vert\leq n$.
    \end{corollary}
    \begin{proof}
       Let us set
       $\psi_n:= \sign(P_n)  \left(\norm{P_n}^{*}_1\right)^{-1}$.
       It follows from Lemma \ref{L6.1} that the function
       $h_{n,s}(x)-\psi_n(-x)$ belongs to $F_n$. Hence,
       \ba
        \dist_\infty(h_{n,s},F_n)
       \leq \norm{\psi_n}_\infty^* = \left(\norm{P_n}^{*}_1\right)^{-1}.
        \ea
        Conversely,
       for any $f\in F_n$,
       \ba
       \norm{h_{n,s} - f}_ \infty^* \geq \left(\norm{P_n}_1^*\right)^{-1} \int_{-\pi}^{\pi} \left(h_{n,s}(-x)- f(-x)\right)\,P_n(x)dx=P_n^{(s)}(0) \left(\norm{P_n}_1^*\right)^{-1},
       \ea
        and the result follows.
    \end{proof}
    Finally, we discuss certain properties of zeros of extremal polynomials.

    \begin{corollary}\label{C6.3}
Let $P_n$ be a polynomial, satisfying \eqref{E6.1} for $s=0$. Then
the following properties are valid:
\begin{itemize}
\item[(a)] $P_n$ has $2n$ simple zeros
$-\pi<\alpha_1<\alpha_2<\cdots< \alpha_{2n}<\pi$;\\
   \item[(b)] for any polynomial $Q_n\in\TT_n$, the following equality holds:
   \beq\label{E6.3}
 2 \sign(P_n(\pi)) \sum_{k=1}^{2n}(-1)^{k+1} Q_n(\alpha_k)  = \norm{P_n}_1^*  Q_n^{'}(0).
   \eeq
   \end{itemize}
    \end{corollary}

    \begin{proof}
    First we recall that $P_n$ is an even polynomial, so $P_n$ does not change its sign at $0$ and $\pi$. Then the number $m$ of zeros from $(-\pi,\pi]$, where $P_n$ changes its sign, is even. Denote the corresponding zeros by
    $\be_k,\,1\le k\le m$, where
    \ba
    -\pi=\be_0<\be_1<\ldots<\be_{m/2}<0<\be_{m/2+1}<\ldots<\be_m<\be_{m+1}=\pi.
    \ea
    Next, by Lemma \ref{E6.1}, the following equalities are valid for any $Q_n\in\TT_n$:
    \bna\label{E6.4}
    \norm{P_n}_1^* Q_n'(0)
    &=&\int_{-\pi}^{\pi} \sign(P_n(x))\, Q'_n(x)dx \\
    &=& \sign(P_n(\pi))\sum_{k=0}^{m} (-1)^k \int_{\beta_k}^{\beta_{k+1}}Q_n'(x)dx\nonumber\\
    &=& \sign(P_n(\pi))\sum_{k=0}^{m} (-1)^k \left(Q_n(\beta_{k+1})-Q_n(\beta_k)\right)\nonumber\\
    &=&2\sign(P_n(\pi))\sum_{k=1}^{m} (-1)^{k+1} Q_n(\beta_{k}).
    \nonumber
     \ena
     Furthermore, we show that $m=2n$. Indeed, let us assume that $m/2<n$. Then the polynomial
     \ba
        Q_{n,0}(x):=\sin x\prod_{k=1}^{m/2}(\cos x-\cos \be_k)
    \ea
    belongs to $\TT_n$ and $Q_{n,0}^\prime(0)=\prod_{k=1}^{m/2}(1-\cos \be_k) \neq 0$. On the other hand,
    $Q_{n,0}(\beta_k)=0$ for $1\leq k\leq m$; therefore,  $Q'_{n,0}(0)=0$ by \eqref{E6.4}. This contradiction shows that
    $m=2n$, i.e., all zeros $\alpha_k=\be_k,\,1\le k\le 2n$, of $P_n$ are real and simple.
    Finally, \eqref{E6.3} follows from \eqref{E6.4}.
    \end{proof}
    Note that for $s=0$ and a special class of even trigonometric polynomials, property (a) was proved in
    \cite[Lemma 6]{AD2022}.

    \vspace{.2cm}
    \section{Appendix: Discrete Hardy Spaces}\label{S7}

First, we recall certain facts from \cite{E1995}  about discrete Hardy spaces.

\begin{definition} (Eoff \cite[p. 509]{E1995})\label{D7.1}
    For a complex-valued sequence $a=(a_k)_{k\in\mathbb{Z}}\in\ell_p(\Z),\,p\in(0,\iy)$, let us define
    \ba
    H(a)_m:=\sum_{k\in \mathbb{Z},k\ne m} \frac{a_k}{m-k},\quad m\in\Z;\qquad
    H(a):=(H(a)_m)_{m\in\mathbb{Z}}.
    \ea

    For $p\in (0,\iy)$, the discrete Hardy space $H_p(\mathbb{Z})$ is the space of all $a\in \ell _p(\Z)$ for which $H(a)\in \ell_p(\Z)$ with the quasi-norm
    \ba
    \norm{a}_{H_p(\Z)} := \norm{a}_{\ell_p(\Z)}+ \norm{H(a)}_{\ell_p(\Z)}.
    \ea
\end{definition}
In addition, for  a complex-valued sequence $a=(a_k)_{k\in \mathbb{Z}}$, let us  define
    \beq\label{E7.1a}
    f_a(x)=\sum_{k\in \mathbb{Z}} (-1)^k a_k \sinc(x-\pi k).
    \eeq
Then the following result holds:
\begin{theorem}\label{T7.2}
The following statements are valid:
    \begin{itemize}
        \item[(a)] for $p\in(1,\iy)$, the map $a\mapsto f_a$ is an isomorphism between $\ell_p(\mathbb{Z})$ and $E_p$;
        \item[(b)] for $p\in(0,1]$, the map $a\mapsto f_a$ is an isomorphism between $H_p(\mathbb{Z})$ and $E_p$.
    \end{itemize}

\end{theorem}
Statement (a) of Theorem \ref{T7.2} was proved by Plancherel and P\'{o}lya \cite{PP1937}
(see also \cite[Theorem 5]{E1995}), while statement (b)
was established by Eoff \cite[Theorem 6]{E1995}.

Second, we state a result about the structure of the discrete Hardy spaces.

\begin{definition} (Boza and Carro \cite[Definition 3.9]{BC1998})\label{D7.2a}
Let $p\in(0,1]$. We say that a complex-valued sequence $a=(a_k)_{k\in \mathbb{Z}}$ is an $H_p$-atom if
                   it satisfies the following properties:
                   \begin{enumerate}
              \item[1)] $\supp a$ is contained in an interval $I$ with the length $\vert I\vert \geq 1$;
              \item[2)] $\norm{a}_{\ell_\infty(\Z)} \leq  |I|^{-\frac{1}{p}}$;
              \item[3)] $\sum_{k \in \mathbb{Z}} a_k k^j=0$ for any integer $j$ such that $0\leq j \leq p^{-1} -1$.
              \end{enumerate}
\end{definition}

In particular, it immediately follows from properties 1) and 2) that for an $H_p$-atom $a$,
\beq\label{E7.1a1}
\norm{a}_{\ell_p(\Z)} \lesssim_p 1, \qquad p\in(0,1].
\eeq

Then the following atomic description for $H_p(\Z)$ holds:

\begin{theorem} (Boza and Carro \cite[Theorems 3.10 and 3.14]{BC1998})\label{T7.3}
           Let $p\in(0,1]$.         Then for any sequence $a\in H_p(\Z)$ we have
          $\norm{a}_{H_p(\Z)}\asymp_p \inf \norm{\lambda}_{\ell_p(\Z_+)}$,
          where $\lambda=\left(\lambda_d\right)_{d\in\Z_+}\in \ell_p(\Z_+)$ and
          the infimum is taken over all possible representations $a=\sum_{d=0}^\iy \lambda_d a(d)$ as a linear combination
of $H_p$-atoms $\{a(d)\}_{d=0}^\iy$.
\end{theorem}
Note that Definitions \ref{D7.1}, \ref{D7.2a} and Theorems \ref{T7.2}, \ref{T7.3} are used in Lemma \ref{L3.6} and in the proof of Lemma \ref{L3.5}.

Third, we define
\ba
    &&H_c(a)_m:=\sum_{k\in \mathbb{Z}} \frac{a_k}{m-k+1/2},\quad m\in\Z;\qquad
    H_c(a):=(H_c(a)_m)_{m\in\mathbb{Z}};\\
    &&\norm{a}_{H_p(\Z)}^* := \norm{a}_{\ell_p(\Z)}+ \norm{H_c(a)}_{\ell_p(\Z)},
    \ea
and show that
     the operator $H_c(a)$ is equivalent to $H(a)$ in the definition of the Hardy space for $p\in(0,1]$.
     This statement is not trivial but
     Eoff \cite[p. 509]{E1995} states this fact without a proof, while using $H_c(a)$ in the proof of statement (b)
      of Theorem \ref{T7.2} as well as in the equivalent definition of $H_p(\mathbb{Z})$.
     That is why we need to prove the following result.
     \begin{proposition}\label{P7.4}
     For any $a\in H_p(\Z),\,p\in(0,1]$, the following relation holds:
     \beq\label{E7.1}
     \norm{a}_{H_p(\Z)}^*\asymp_p \norm{a}_{H_p(\Z)}.
     \eeq
     \end{proposition}

          \begin{proof}
          Note first that the proof of the $\gtrsim$ part of the proposition is based on
          the following discretization inequalities of Plancherel and P\'{o}lya
          \cite[p. 110]{PP1937} (see also
          \cite[Theorems 2 (ii) and 4 (i)]{E1995}):
          \bna
          &&\left(\sum_{n\in\Z} |f(\pi n)|^p \right)^{{1}/{p}}
          \lesssim_p \norm{f}_p,\qquad  f\in E_{p},\quad p>0;\label{E7.1b}\\
          &&\left(\sum_{n\in\Z} |f(\pi n)|^p \right)^{{1}/{p}}
          \gtrsim_{p,\tau} \norm{f}_p,\qquad
          \left(f(\pi n)\right)_{n\in \Z}\in \ell_p(\Z),
          \quad p>0;\label{E7.1c}
          \ena
          where $f$ in \eqref{E7.1c}
           is an entire function of exponential type $\tau\in (0,1)$.

          Assume that $\norm{a}_{H_p(\Z)}^*<\iy$.
          Then $f_a$, defined by \eqref{E7.1a}, belongs to $E_p$ and
          \beq\label{E7.1d}
          \left\|f_a(\cdot/2)\right\|_p\lesssim_p \norm{a}_{H_p(\Z)}^*.
          \eeq
          Indeed, since $a\in \ell_2(\Z)$, then $\left\|f_a\right\|_2<\iy$ and $\mbox{supp}\, \widehat{f_a}\in[-1,1]$. Therefore, $f_a(\cdot/2)\in E_{1/2,2}$.
          In addition,
          \ba
          f_a(\pi n/2)=\left\{\begin{array}{ll}
          (-1)^m a_m,&n=2m,\\
          (-1)^m\pi^{-1}H_c(a)_m,& n=2m+1,
          \end{array}\right.,\qquad m\in\Z.
          \ea
          Hence
          $\left(\sum_{n\in\Z}\left\vert f_a(\pi n/2)\right\vert^p\right)^{1/p}\lesssim_p
          \norm{a}_{H_p(\Z)}^*$.
          Thus \eqref{E7.1d} follows from \eqref{E7.1c}
          for $\tau=1/2$ and $f=f_a(\cdot/2)$ and, in addition, $f_a\in E_p$
          (cf. \cite[Proof of Theorem 6]{E1995}).

Next, using Bernstein's inequality (see \eqref{E1.2} for $s=1$ and $p=q$) and, in addition,
applying inequalities \eqref{E7.1b} and \eqref{E7.1d} to $f_a^\prime$ and $f_a$, respectively, we have

              \beq\label{E7.2}
              \norm{H(a)}_{\ell_p(\Z)} = \pi\left(\sum_{m\in\Z} |f^\prime_a(\pi m)|^p \right)^{{1}/{p}}
              \lesssim_p \norm{f_a^\prime}_p
              \lesssim_p \norm{f_a}_p
              \lesssim_p \norm{a}^*_{H_p(\Z)}.
              \eeq

              To prove the $\lesssim$ part of the proposition, we first note that the following estimate for the norm of $H_c(a)$ established in \cite[Lemme 8, p. 136]{PP1937}:
               \beq\label{E7.2a}
               \norm{H_c(a)}_{\ell_p(\Z)}\lesssim_p \norm{a}_{\ell_p(\Z)},\qquad p>1,
               \eeq
               is not valid for $p\in(0,1]$. That is why
              we first estimate $\norm{H(a)}_{\ell_p(\Z)},\,p\in(0,1],$ for  an $H_p$-atom $a$ (see Definition \ref{D7.2a}) and then discuss any sequence from $H_p(\Z)$.
              Without loss of generality we can assume that $\mbox{supp}\, a\subseteq [-N,N]\cap\Z,\,N\in\N$,
              with
              \bna
              &&\norm{a}_{\ell_\iy(\Z)}\lesssim_p N^{-1/p},\label{E7.2b}\\
              &&\sum_{k=-N}^N a_k k^j=0,\qquad 0\leq j \leq j_0,\quad j_0:=\lfloor 1/p-1\rfloor.\label{E7.3c}
              \ena
              Then using H\"{o}lder's inequality, estimate \eqref{E7.2a} for $p=2$,
              and \eqref{E7.2b}, we obtain
              \beq\label{E7.2c}
               \sum_{m=-2N}^{2N} |H_c(a)_m|^p
               \lesssim_p N^{1-  {p}/{2}}  \norm{H_c(a)}_{\ell_2(\Z)}^p
               \lesssim N^{1-  {p}/{2}}  \norm{a}_{\ell_2(\Z)}^p
               \lesssim_p 1.
              \eeq
              Next, we show that
               \beq\label{E7.2d}
            \sum_{\vert m\vert=2N+1}^\infty |H_c(a)_m|^p \lesssim_p 1.
            \eeq
  Indeed, using \eqref{E7.3c}, we obtain for $m\geq 2N+1$,
\ba
      H_c(a)_m&=&\sum_{k=-N}^N\frac{a_k}{m -(k- 1/2)}
              = m^{-1} \sum_{k=-N}^N  a_k  \sum_{j=0}^\infty \left(\frac{k-1/2}{m}\right)^j\\
              &=&m^{-1} \sum_{j=0}^\infty \sum_{k=-N}^N  a_k   \left(\frac{k-1/2}{m}\right)^j
               =m^{-1} \sum_{j=j_0+1}^\infty \sum_{k=-N}^N  a_k   \left(\frac{k-1/2}{m}\right)^j\\
              &=&  m^{-1}  \sum_{k=-N}^N  a_k
                 \sum_{j=j_0+1}^\infty\left(\frac{k-1/2}{m}\right)^j
               =  m^{-j_0-1}\sum_{k=-N}^N  \frac{a_k}{m -(k- 1/2)} (k-1/2)^{j_0+1},
\ea
where the series above are convergent.
            Hence, taking account of \eqref{E7.2b}, we have
            \beq\label{E7.2e}
            \left\vert H_c(a)_m\right\vert \lesssim_p m^{-j_0-2} N^{j_0+2-1/p},\qquad m\geq 2N+1,\quad N\in\N.
            \eeq
            It follows from \eqref{E7.2e} that
            $\sum_{m=2N+1}^\infty |H_c(a)_m|^p \lesssim_p 1$
            since $p(j_0+2)=p(\lfloor 1/p\rfloor+1) > 1$. The estimate
             $\sum_{m=-2N-1}^{-\infty} |H_c(a)_m|^p \lesssim_p 1$
             can be proved similarly.
             Thus, \eqref{E7.2d} is established.
            Combining together \eqref{E7.2c} with \eqref{E7.2d}, we see that for an $H_p$-atom $a$,
            \beq\label{E7.3a}
            \norm{H_c(a)}_{\ell_p(\Z)}\lesssim_p 1.
            \eeq

             Furthermore, by Theorem \ref{T7.3},  any sequence $a\in H_p(\Z)$ has an atomic decomposition $a=\sum_{d=0}^\infty \lambda _d a(d)$, where $a(0),\,a(1),\ldots,$ are $H_p$-atoms,
 with $\norm{a}_{H_p(\Z)} \asymp_p \norm{\lambda}_{\ell_p(\Z_+)}$.
 In addition, we have by \eqref{E7.1a1} that
            $\sum_{d=0}^\iy \sum_{k\in\Z}|\lambda _d| |a_k(d)|\lesssim_p \norm{\lambda}_{\ell_1(\Z_+)}\leq \norm{\lambda}_{\ell_p(\Z_+)}<\infty$.
            Therefore,
            \beq\label{E7.3b}
            H_c\left(\sum_{d=0}^\infty \lambda_d a(d)\right)_m
            =\sum_{d=0}^\infty\lambda_d  H_c\left(a(d)\right)_m.
            \eeq
            Thus, using \eqref{E7.3a}, \eqref{E7.3b}, and Theorem \ref{T7.3} we arrive at the relations
            \beq\label{E7.4}
            \norm{H_c(a)}^p_{\ell_p(\Z)}
             \leq \inf_{\lambda}\sum_{d=0}^\iy \lambda_d ^p \norm{H_c(a(d))}_{\ell_p(\Z)}^p
             \lesssim_p \inf_{\lambda}\sum_{d=0}^\infty \lambda_d ^p
             \lesssim_p \norm{a}_{H_p(\Z)}^p ,
            \eeq
            where the infima in \eqref{E7.4} are taken over all possible representations
            $a=\sum_{d=0}^\iy \lambda_d a(d)$.
            Finally, \eqref{E7.1} follows from \eqref{E7.2} and \eqref{E7.4}.
          \end{proof}
          \noindent
\textbf{Acknowledgements.} We are grateful to Roald M. Trigub  for providing reference
\cite{BH2020} and to Kristina Oganesyan for her valuable remarks.

\end{document}